\documentclass[10pt]{amsart}
\usepackage{subfigure}
\usepackage{graphicx}
\usepackage{enumerate}
\usepackage{bbm}
\usepackage[
colorlinks=true, urlcolor=blue]{hyperref}

\newtheorem{thm}{Theorem}
\newtheorem{thm*}{Theorem}

\newtheorem{prop}[thm]{Proposition}

\theoremstyle{remark}
\newtheorem{remark}[thm]{Remark}

\theoremstyle{definition}

\newcommand{\R}{\mathbbm{R}}
\newcommand{\rd}{\mathrm{d}}

\usepackage{color}

\title{Riesz energies and the magnitude of manifolds}
\author{Heiko Gimperlein, Magnus Goffeng}
\address{Heiko Gimperlein\newline
\indent Engineering Mathematics\newline
\indent Universit\"{a}t Innsbruck\newline
\indent Technikerstraße 13 \newline
\indent 6020 Innsbruck\newline
\indent Austria \newline\newline
\indent Magnus Goffeng,\newline
\indent Centre for Mathematical Sciences\newline 
\indent Lund University \newline 
\indent Box 118, SE-221 00 Lund\newline 
\indent Sweden\newline
}
\email{heiko.gimperlein@uibk.ac.at, magnus.goffeng@math.lth.se}
\begin{document}
\maketitle

\begin{abstract}
We study the geometric significance of Leinster's magnitude invariant. For closed manifolds we find a precise relation with Brylinski's beta function and therefore with classical invariants of knots and submanifolds. In the special case of compact homogeneous spaces we obtain an elementary proof that the residues of the beta function contain the same geometric information as the asymptotic expansion of the magnitude function. For general closed manifolds we use the recent pseudodifferential analysis of the magnitude operator to relate these via an interpolating polynomial family. Beyond manifolds, the relation with the Brylinski beta function allows to deduce unexpected properties of the magnitude function for the $p$-adic integers.
\end{abstract}

\section{Introduction}

Originally introduced in the context of (finite) enriched categories, the magnitude invariant has been shown to unify notions of ``size'' like the cardinality of a set, the length of an interval or the Euler characteristic of a triangulated manifold. An overview is given in \cite{leinmeck}. 
As a special case of the results in the current paper, we relate magnitude to the Riesz energy of a knot \cite{bryl99,ohara91,ohara18,ohara23}, an invariant also known as M\"{o}bius energy as it is famously preserved by M\"{o}bius transformations \cite{fhw94}.

The magnitude of metric spaces was proposed and studied by Leinster and Willerton, who viewed a metric space as a category enriched over $[0,\infty)$ \cite{leinster, leinwill}: A weight function on a finite metric space $(X,d)$ is a function $w : X \to \R$ which satisfies $\sum_{y \in X} \mathrm{e}^{-\rd(x,y)}w(y) =1$ for all $x \in X$. Given any weight function $w$, the magnitude of $X$ is defined as $\mathrm{Mag}(X) := \sum_{x \in X} w(x)$. Beyond the magnitude of an individual space $(X,d)$, it proves fruitful to study the magnitude function $\mathcal{M}_X(R) := \mathrm{Mag}(X, R\cdot d)$ for $R>0$. The magnitude invariant has been extended from finite metric spaces to compact metric spaces $(X,\rd)$ \cite{meckes,meckes2}, such as domains in $\mathbf{R}^n$ \cite{barcarbs,gimpgoff,gimpgoff2,leinwill}, homogeneous spaces \cite{will14}, or manifolds with boundary \cite{gimpgofflouc,gimpgofflouctech}.

A main open question about magnitude is which geometric properties of $X$ it determines. We here study the relationship between $\mathcal{M}_X(R)$ and objects defined from the Riesz energy, such as Brylinski beta functions and residues of manifolds. Similarities were first noted in \cite{gimpgofflouc}. In this article we find a one-to-one relation between these geometric invariants, as special cases of a one-parameter family $\mathcal{M}_X(R,\nu)$ of invariants. At $\nu=-1$, we recover the magnitude function $\mathcal{M}_X(R,-1)$, while at $\nu =+1$ the inverse Mellin transform of $\mathcal{M}_X(R,+1)$ is the Brylinski beta function. 

Our main result is deduced from the asymptotics for $R \to \infty$, which for an $n$-dimensional manifold $X$  take the form
$$\mathcal{M}_X(R,\nu)\sim \sum_{j=0}^\infty c_j(X,\nu)R^{n-j},$$ 
in combination with well-known relations between asymptotic expansions and {the poles of their Mellin transforms (see Theorem \ref{thm2})}.
For any $k$, we then show that the expansion coefficients $\{c_j(X,-1)\}_{j=0}^k$ of the magnitude function uniquely determine the coefficients $\{c_j(X,+1)\}_{j=0}^k$ of the beta function, and vice versa (see Theorems \ref{mainthm2} and \ref{thm3} for more general statements).

We first present this relation for compact homogeneous spaces $X=G/H$ of a locally compact group $G$ by a stabilizer subgroup $H$, endowed with an invariant distance function (see Theorems \ref{thm2} and \ref{mainthm2}). Based on the pseudodifferential analysis of the magnitude operator in \cite{gimpgofflouctech}, we then discuss extensions to classes of closed manifolds (see Theorem \ref{thm3}). 
As an application of our methods, we deduce properties of the magnitude function for the $p$-adic integers $X=\mathbb{Z}_p$, a compact group. We also illustrate the explicit relation between magnitude and residues for the sphere $X=S^n$, where both invariants have been studied separately \cite{ohara23,will14}.\\

\noindent {\bf Acknowledgements:} The authors thank Tom Leinster for fruitful discussions and Simon Willerton for comments on an early version of the manuscript. The second listed author was supported by the Swedish Research Council Grant VR 2018-0350.\\

\noindent \emph{Notation:} In this paper when we write  $\mathrm{Re}(R)\gg0$ or $\mathrm{Re}(R)\to \infty$, we tacitly assume that $R$ eventually belongs to a small sector around the positive real axis. By \cite{gimpgofflouctech}, for a manifold $X$ of dimension $n$ we can use the sector of opening angle $\pi/(n+1)$.

\section{Residues of manifolds}
\label{residuesec}

A well-studied invariant of a knot is the M\"{o}bius energy \cite{ohara91,ohara18}, which was proven to be M\"{o}bius invariant by Freedman-He-Wang \cite{fhw94}. More generally, one considers the Brylinski beta function \cite{bryl99,FV}, also known as the regularized Riesz $z$-energy or the meromorphic energy function. For a compact metric measure space $(X,\rd,\mu)$, it is the function of the complex variable $z$ defined by
\begin{equation}
\label{brylzeta}
B_X(z):=\int_X\int_X \rd(x,y)^z\rd\mu(x)\rd\mu(y).
\end{equation}
A priori, $B_X$ is a holomorphic function for $\mathrm{Re}(z)>0$. When $\rd^2$ is smooth near $x=y$, e.g.~for a Riemannian manifold, $B_X$ extends meromorphically to $\mathbb{C}$ by standard arguments from distribution theory. We define the residue function by 
$$R_X(z):=\mathrm{Res}_{s=z}B_X(s).$$
The domain of $R_X$ is the maximal connected domain of meromorphicity of $B_X$, where $R_X$ has locally finite support. 

To study $B_X$ and $R_X$, it proves useful 
to introduce the operator 
\begin{equation}
\label{magop}
\mathcal{Z}_X(R)f(x):=\int_X \mathrm{e}^{-R\rd(x,y)}f(y)\rd \mu(y),
\end{equation}
defined for suitable distributions $f$. As noted in  \cite[Section 3.5]{gimpgofflouc}, the support and the non-zero values of $R_X$ are determined from the asymptotic behavior as $\mathrm{Re}(R)\to \infty$ of the entire function $\mathfrak{m}_X:\mathbb{C}\to \mathbb{C}$ defined by
$$\mathfrak{m}_X:R\mapsto \langle \mathcal{Z}_X(R) 1,1\rangle_{L^2(X,\mu)} = \int_X\int_X \mathrm{e}^{-R\rd(x,y)}\rd\mu(x)\rd\mu(y).$$ 
Indeed, the (flipped) Brylinski beta function $B_X(-z)=\int_0^\infty \mathfrak{m}_X(R)R^{z-1}\rd R$ is the Mellin transform of $\mathfrak{m}_X$, and therefore the pole structure of $B_X$ relates to the asymptotic behaviour of $\mathfrak{m}_X$ as $R\to +\infty$, as we recall in the Appendix. If $X$ is a Riemannian manifold with boundary, $B_X$ and $R_X$ encode geometric information such as the volume, the measure of the boundary or integrals of curvatures \cite{ohara18,ohara23}.

\section{Magnitude and complex powers for compact geometries}
\label{lknland}

More recently, the magnitude function $\mathcal{M}_X(R)$ was introduced and studied for certain compact metric spaces. Starting from the magnitude operator $\mathcal{Z}_X(R)$ defined above, the magnitude function takes the form
\begin{equation}\label{magZ}\mathcal{M}_X(R):=\langle \mathcal{Z}_X(R)^{-1}1,1\rangle_{L^2(X,\mu)}.\end{equation}
Here, the inverse $\mathcal{Z}_X(R)^{-1}$ is interpreted appropriately \cite{gimpgofflouctech}. Using \eqref{magZ}, the magnitude operator $\mathcal{Z}_X(R)$ determines the magnitude function $\mathcal{M}_X(R)$ in all cases relevant to this paper \cite{gimpgofflouc,gimpgofflouctech}. More generally, the magnitude of a positive definite, compact metric spaces $X$ is defined as the supremum of the magnitudes of finite metric subspaces of $X$ \cite{meckes}.

The magnitude operator has previously been studied on different function spaces \cite{gimpgofflouc,meckes,meckes2,will14}. In particular, we note that for a closed $n$-dimensional manifold $X$, it defines an isomorphism $\mathcal{Z}_X(R):H^{-\frac{n+1}{2}}(X)\to H^{\frac{n+1}{2}}(X)$ when $\mathrm{Re}(R)\gg0$, provided that the distance function is sufficiently nice (e.g.~for subspace distances of submanifolds in Euclidean or hyperbolic space, see \cite{gimpgofflouc,gimpgofflouctech}). 

For the purposes of this paper, it suffices to consider $\mathcal{Z}_X(R)$ on $L^2(X,\mu)$ where it is a self-adjoint and compact operator as soon as $R>0$, because the kernel is continuous and symmetric. From functional analytic arguments, we therefore deduce that for $R>0$, the magnitude operator \eqref{magop} admits complex powers $\mathcal{Z}_X(R)^\nu$, $\nu\in \mathbb{C}$, with $1\in \cap_{\nu\in \mathbb{C}} \mathrm{Dom}(\mathcal{Z}_X(R)^\nu)$. This means that there exists a family of operators $\mathcal{Z}_X(R)^\nu$, $\nu\in \mathbb{C}$, satisfying the semigroup property $\mathcal{Z}_X(R)^{\nu_1}\mathcal{Z}_X(R)^{\nu_2}=\mathcal{Z}_X(R)^{\nu_1+\nu_2}$, with  $\mathcal{Z}_X(R)^1= \mathcal{Z}_X(R)$, $\mathcal{Z}_X(R)^{-1}= \mathcal{Z}_X(R)^{-1}$ (on $\overline{\mathcal{Z}_X(R)L^2(X,\mu)}$), and each  $\mathcal{Z}_X(R)^\nu$ is densely defined with $\mathrm{Dom}(\mathcal{Z}_X(R)^\nu)=L^2(X,\nu)$ for $\mathrm{Re}(\nu)\geq 0$.

Let us illustrate this discussion in the  case, where $X=G/H$ is a homogeneous space for a compact group $G$ with stabilizer (subgroup) $H$. We assume that $\rd$ is a $G$-invariant distance function on $X$ and $\mu$ is the normalized $G$-invariant measure on $X$. Since $G$ is compact, we can decompose $L^2(X,\mu)$ via representation theory as
\begin{equation}
\label{decom}
L^2(X,\mu)=\bigoplus_{\pi\in \hat{G}} \mathcal{H}_\pi^{\oplus n_\pi},
\end{equation}
where $\hat{G}$ denotes the discrete set of unitary equivalence classes of unitary representations $\pi:G\to U(\mathcal{H}_\pi)$ of $G$, and $n_\pi\in \mathbb{N}$ denotes the multiplicity of $\pi$ in $L^2(X,\mu)$. Each of the Hilbert spaces $\mathcal{H}_\pi$ is  finite-dimensional, because $G$ is compact. We write $\mathcal{H}_\pi^{\oplus n_\pi}=\mathbb{C}^{n_\pi}\otimes \mathcal{H}_\pi$ and identify the space of $G$-equivariant operators on $\mathcal{H}_\pi^{\oplus n_\pi}$ with the space $M_{n_\pi}(\mathbb{C})$ of $n_\pi\times n_\pi$-matrices. Since $\mathcal{Z}_X(R)$ is a $G$-equivariant compact operator on $L^2(X,\mu)$, under the decomposition \eqref{decom} this operator decomposes as 
$$\mathcal{Z}_X(R)=\bigoplus_{\pi\in \hat{G}} \mathfrak{Z}_\pi(R)\otimes 1_{\mathcal{H}_\pi}.$$
Here, $\mathfrak{Z}_\pi:\mathbb{C}\to M_{n_\pi}(\mathbb{C})$ is an entire function and for each $R\in \mathbb{C}$, $\|\mathfrak{Z}_\pi(R)\|\to 0$ as $\pi\to \infty$ in $\hat{G}$. We now define the complex powers of $\mathcal{Z}_X(R)$ using the functional calculus for matrices,
$$\mathcal{Z}_X(R)^\nu=\bigoplus_{\pi\in \hat{G}} \mathfrak{Z}_\pi(R)^\nu\otimes 1_{\mathcal{H}_\pi}.$$
A similar argument allows us to construct complex powers of $\mathcal{Z}_X(R)$ in the more general setting of a compact homogeneous space $G/H$, where $G$ is not necessarily compact, if we replace the use of representation theory by the spectral theory for compact operators. A relevant such example is given by $G=SL(2,\mathbb{R})$ and $H$ a surface group.

The discussion in Section \ref{residuesec} (that will be made precise below in Theorem \ref{mainthm2}) shows that the residues of a manifold $X$ are determined by the asymptotic behavior of $\mathfrak{m}_X(R) = \langle \mathcal{Z}_X(R)1,1\rangle_{L^2(X,\mu)}$. On the other hand, the magnitude function is computed as $\mathcal{M}_X(R) = \langle \mathcal{Z}_X(R)^{-1}1,1\rangle_{L^2(X,\mu)}$. The next result interpolates these two invariants.
\begin{thm}
\label{thm1}
Let $(X,\rd,\mu)$ be a compact metric measure space of the following form:
\begin{itemize}
\item $X=G/H$ is a homogeneous space, $\rd$ a $G$-invariant distance function and $\mu$ a normalized invariant measure; or
\item $X$ is a compact smooth manifold, $\rd$ a distance function with $\rd^2$ smooth on $X\times X$ and regular at the diagonal (see \cite[Definition 2.2]{gimpgofflouctech}) and $\mu$ the volume density defined from $\rd^2$ at the diagonal (as in \cite[Section 2]{gimpgofflouctech}).
\end{itemize}
Then the function 
$$\mathcal{M}_X(R,\nu):=\langle \mathcal{Z}_X(R)^\nu1,1\rangle_{L^2(X,\mu)}$$
is holomorphic for $\mathrm{Re}(R)\gg 0$ and $\nu\in \mathbb{C}$. 
\end{thm}

\begin{proof}
We prove the two cases separately. We start to consider the case of $X$ a compact homogeneous space. For notational simplicity we assume that $X=G/H$ for $G$ compact. We write $\pi_0:G\to U(\mathbb{C})$ for the trivial $G$-representation $\pi_0(g)=1$ in dimension $1$. We note that $n_{\pi_0}=1$ and the constant function $1\in L^2(X,\mu)$ is an orthonormal basis for the summand $\mathcal{H}_{\pi_0}=\mathbb{C}$ in the decomposition of Equation \eqref{decom}. Therefore, with $eH\in G/H$ the coset of the identity element $e\in G$,
$$\mathfrak{Z}_{\pi_0}(R)=\langle \mathcal{Z}_X(R)1,1\rangle_{L^2(X,\mu)}=\int_X \mathrm{e}^{-R\rd(eH,y)}\rd\mu(y),$$
which defines an entire function $\mathbb{C}\to M_1(\mathbb{C})=\mathbb{C}$. It is clear from the construction that $\mathcal{M}_X(R,\nu)=\langle \mathcal{Z}_X(R)^\nu1,1\rangle_{L^2(X,\mu)}=\mathfrak{Z}_{\pi_0}(R)^\nu$ is holomorphic for $\mathrm{Re}(R)\gg 0$ and $\nu\in \mathbb{C}$.

The detailed proof of Theorem \ref{thm1} for closed manifolds is technically more involved. It is an immediate consequence of the following series of observations, for which we refer to the literature. Firstly, if $\rd$ is a distance function with $\rd^2$ smooth on $X\times X$ and regular at the diagonal it has Property (SMR) (in the sense of \cite[Defintion 3.3]{gimpgofflouctech}), using the argument in \cite[Proposition 3.4]{gimpgofflouctech}. Secondly, we note that then $R^{-1}\mathcal{Z}_X(R)$ by definition is an elliptic pseudodifferential operator with parameter $R$, on which it depends holomorphically, see \cite[Corollary 3.5]{gimpgofflouctech}. Thirdly, following standard techniques as in \cite[Chapter II]{shubinbook}, the preceding observation implies that the complex powers $R^{-\nu}\mathcal{Z}_X(R)^\nu$ form a holomorphic family in $\nu$ of elliptic pseudodifferential operators with parameter $R$, on which it depends holomorphically. This implies that $\langle \mathcal{Z}_X(R)^\nu1,1\rangle_{L^2(X,\mu)}$ is holomorphic for $\mathrm{Re}(R)\gg 0$ and $\nu\in \mathbb{C}$. 
\end{proof}

\section{Homogeneous spaces}
\label{sectionhomos}

Let $X=G/H$ be a compact homogeneous space, $\rd$ an invariant distance function and $\mu$ the normalized invariant measure. From the proof of Theorem \ref{thm1} we conclude the following.

\begin{prop}
\label{mxandmx}
The function $\mathcal{M}_X(R,\nu)$ takes the form 
$$\mathcal{M}_X(R,\nu)=\mathfrak{m}_X(R)^\nu,$$
where 
$$\mathfrak{m}_X(R)=\langle \mathcal{Z}_X(R)1,1\rangle_{L^2(X,\mu)}=\int_X \mathrm{e}^{-R\rd(eH,y)}\rd\mu(y).$$
\end{prop}

The argument leading to Proposition \ref{mxandmx}  closely relates to the argument in \cite{will14}. The case $\nu=-1$ is called Speyer's homogeneous magnitude theorem in \cite[Theorem 1]{will14}. We combine this line of thought with  Proposition \ref{relarlandlna} in the Appendix, which we apply to the function $e=\mathfrak{m}_X$. The assumptions from Proposition \ref{relarlandlna} are satisfied near $R=0$, because $\mathfrak{m}_X$ is entire with Taylor series 
\begin{equation}
\label{mxatzero}
\mathfrak{m}_X(R)=\sum_{j=0}^\infty B_j R^j, \quad B_j=\frac{(-1)^j}{j!} \int_X \rd(eH,y)^j\rd\mu(y).
\end{equation}
We therefore conclude the following statement for compact homogeneous spaces.

%

\begin{thm}
\label{thm2}
Let $X=G/H$ be a compact homogeneous space, $\rd$ an invariant distance function and $\mu$ a normalized invariant measure. Further, let $\gamma\in \R$ and $N\in \mathbb{N}$. For the function $\mathcal{M}_X(R,\nu)=\mathfrak{m}_X(R)^\nu$ the following are equivalent:
\begin{enumerate}
\item $\mathfrak{m}_X$ admits an asymptotic expansion as $\mathrm{Re}(R)\to +\infty$
$$\mathfrak{m}_X(R)=\sum_{j=0}^Na_jR^{\gamma-j}+O(R^{\gamma-N-1}), \qquad \quad \text{as $R\to +\infty$}.$$
\item Uniformly for $\nu$ in any compact subset of $\mathbb{C}$ we have  asymptotic expansions 
$$\mathcal{M}_X(R,\nu)=\sum_{j=0}^N \alpha_j(\nu)R^{\gamma \nu-j}+O(R^{\gamma\nu-N-1}), \quad \mbox{as $\mathrm{Re}(R)\to \infty$}.$$
\item For some $\nu\neq 0$ we have an asymptotic expansion 
$$\mathcal{M}_X(R,\nu)=\sum_{j=0}^N \alpha_j(\nu)R^{\gamma \nu-j}+O(R^{\gamma\nu-N-1}), \quad \mbox{as $\mathrm{Re}(R)\to \infty$}.$$
\item The Brylinski beta function \eqref{brylzeta} extends meromorphically to the half-plane $\mathrm{Re}(z)> \gamma-N-1$ with simple poles, and for any $M$ we have 
$$\Gamma(-z)B_X(z)=\sum_{j=0}^N \frac{b_j}{z-\gamma+ j} + \sum_{j=0}^M \frac{B_j}{z- j}+f_{N,M}(z),$$
where $f_{N,M}$ is holomorphic in the strip $\gamma-N-1<\mathrm{Re}(z)<M+1$ and $B_j=(-1)^jB_X(j)/j!$ as in Equation \eqref{mxatzero}.
\end{enumerate}
\end{thm}

\begin{remark}
\label{aldnlknad}
If any of the equivalent statements of Theorem \ref{thm2} holds, then $a_j=b_j$ and 
$$B_X(z)=\sum_{j=0}^N \frac{1}{\Gamma(-\gamma+j)}\frac{b_j}{z-\gamma+ j} + \tilde{f}_N(z),$$  
where $\tilde{f}_{N}$ is holomorphic in the half-plane $\mathrm{Re}(z)> \gamma-N-1$. We shall see below that when $X=G/H$ is a smooth compact manifold of dimension $n$, the statements of Theorem \ref{thm2} hold for $\gamma=-n$ and any $N$. Note, in particular, that we have re-proven the well known result that the Riesz energy of manifolds has poles at $-n,-n-1,-n-2,\dots$.
\end{remark}

Let us describe in more detail how the coefficients in the asymptotic expansion of $\mathcal{M}_X(R,\nu)=\mathfrak{m}_X(R)^\nu$ depend on $\nu$.

\begin{prop}
\label{mobvingpoweow}
Assume that any of the equivalent statements in Theorem \ref{thm2} holds. Then the coefficients $\alpha_j(\nu)$ are computed from $a_0,\ldots, a_j$ by 
\begin{align}
\label{knklnlknad}
\alpha_j(\nu)=&a_0^\nu g_j\left(\nu,\frac{a_1}{a_0},\frac{a_2}{a_0},\ldots, \frac{a_j}{a_0}\right),
\end{align}
where $g_j=g_j(t_1,\ldots, t_j)$ are polynomials of degree $j$ in $\nu$ and homogeneous of degree $j$ in $(t_1,t_2,\ldots, t_j)$ (for $t_j$ having degree $j$) which are determined by the formal Taylor expansion
$$\left(1+\sum_{l=1}^{\infty} t_lx^{l}\right)^\nu=\sum_{j=0}^\infty g_j(\nu,t_1,\ldots, t_j)x^j.$$
We have $\alpha_j(1)=a_j$ and for $\nu,\nu'\neq 0$ we obtain $\alpha_j(\nu+\nu')$ from $\alpha_0(\nu'), \alpha_1(\nu'), \ldots, \alpha_j(\nu')$ as 
$$\alpha_j(\nu+\nu')=\alpha_0(\nu')^\nu g_j\left(\nu,\frac{\alpha_1(\nu')}{\alpha_0(\nu')},\frac{\alpha_2(\nu')}{\alpha_0(\nu')},\ldots, \frac{\alpha_j(\nu')}{\alpha_0(\nu')}\right).$$
\end{prop}

\begin{proof}
Note that for given $\nu\neq 0$ the values of $(t_1,t_2,\ldots,t_j)$ are determined by the values $(g_j(\nu,t_1,\ldots, t_k))_{k=1}^j$. As the following arguments show, for a fixed $\nu\neq0$ the map 
$$G_{j,\nu}:\R^j\to \R^j, (t_1,\ldots, t_j)\mapsto (g_j(\nu,t_1,\ldots, t_k))_{k=1}^j,$$
is a polynomial diffeomorphism with inverse $G_{j,\nu}^{-1}=G_{j,\nu^{-1}}$. In fact, for $\nu_1,\nu_2\in \mathbb{C}^\times$ we have 
$$G_{j,\nu_1}\circ G_{j,\nu_2}=G_{j,\nu_1\nu_2}.$$
This identity follows from uniqueness of Taylor polynomials and the identities
\begin{align*}
\left(1+\sum_{k=1}^j t_kx^k\right)^{\nu_1\nu_2}+&O(x^{j+1})=\left(1+\sum_{k=1}^j g_k(\nu_2, t_1,\ldots, t_k)x^k\right)^{\nu_1}+O(x^{j+1})=\\
&=1+\sum_{k=1}^j g_k(\nu_1,g_1(\nu_2, t_1), \ldots g_k(\nu_2,t_1,\ldots, t_k))x^k+O(x^{j+1}),\\
\left(1+\sum_{k=1}^j t_kx^k\right)^{\nu_1\nu_2}+&O(x^{j+1})=1+\sum_{k=1}^j g_k(\nu_1\nu_2, t_1,\ldots, t_k)x^k+O(x^{j+1}).
\end{align*}
\end{proof}

\begin{remark}
The polynomial $g_j=g_j(t_1,\ldots, t_j)$ is given by
$$g_j(\nu,t_1,\ldots, t_j)=\sum_{\substack{\mathbbm{k}=(k_1,\ldots,k_j)\in \mathbb{N}^{j},\\ \sum_l lk_l=j}}\frac{\nu(\nu-1)\cdots (\nu-\sum_{l=1}^j k_l+1)}{\prod_{l=1}^j k_l!}\prod_{l=1}^j t_l^{k_l}.$$
For the first few values of $j$ we find 
\begin{align*}
g_0(\nu)=&1,\\
g_1(\nu,t_1)=&\nu t_1,\\
g_2(\nu,t_1,t_2)=& \frac{\nu(\nu-1)}{2}t_1^2+\nu t_2,\\
g_3(\nu,t_1,t_2,t_3)=& \frac{\nu(\nu-1)(\nu-2)}{6}t_1^3+\nu(\nu-1)t_1t_2+\nu t_3.
\end{align*}
\end{remark}

Proposition \ref{mobvingpoweow} allows to move back and forth between asymptotic expansions of $\mathcal{M}_X(R,\nu)$ for different $\nu$. We present this conclusion as a relation between the magnitude function $\mathcal{M}_X(R)$ and its generalization $\mathcal{M}_X(R,\nu)$.

\begin{thm}
\label{mainthm2}
Let $X=G/H$ be a compact homogeneous space, $\rd$ an invariant distance function and $\mu$ a normalized invariant measure. Take $\gamma\in \R$ and $N\in \mathbb{N}$. The following are equivalent: 
\begin{enumerate}
\item There exists an asymptotic expansion as $\mathrm{Re}(R)\to +\infty$
$$\mathcal{M}_X(R)\equiv \mathfrak{m}_X(R)^{-1}=\sum_{j=0}^Nc_jR^{-\gamma-j}+O(R^{\gamma-N-1}).$$
\item Uniformly for $\nu$ in any compact subset of $\mathbb{C}$ there exists an asymptotic expansion as $\mathrm{Re}(R)\to +\infty$
$$\mathcal{M}_X(R,\nu)=\sum_{j=0}^N \alpha_j(\nu)R^{\gamma \nu-j}+O(R^{\gamma\nu-N-1}).$$
\end{enumerate}
If any of the above statements holds, then $c_0,c_1,\ldots ,c_j$ polynomially determine $\alpha_0(\nu), \alpha_1(\nu), \ldots, \alpha_j(\nu)$ via 
$$\alpha_j(\nu)=c_0^{\nu+1} g_j\left(\nu+1,\frac{c_1}{c_0},\frac{c_2}{c_0},\ldots, \frac{c_j}{c_0)}\right).$$
Conversely, $\alpha_0(\nu), \alpha_1(\nu), \ldots, \alpha_j(\nu)$ polynomially determine $c_0, c_1,\ldots, c_j$ from the formulas
$$c_j=\alpha_0(\nu)^{-\nu-1} g_j\left(-\nu-1,\frac{\alpha_1(\nu)}{\alpha_0(\nu)},\frac{\alpha_2(\nu)}{\alpha_0(\nu)},\ldots, \frac{\alpha_j(\nu)}{\alpha_0(\nu)}\right).$$
\end{thm}

\section{Expansion coefficients on smooth manifolds}

We now consider the existence of the asymptotic expansions assumed in Section \ref{sectionhomos} beyond homogeneous spaces. Based on the results of \cite{gimpgofflouctech}, the proof above generalizes to closed manifolds under some technical assumptions on the distance function. This approach provides a structural description of the coefficients in the asymptotic expansion and suggests computational tools to determine them. 

\begin{thm}
\label{thm3}
Let $X$ be an $n$-dimensional compact smooth manifold and $\rd$ a distance function on $X$. Assume that $\rd^2$ is regular near the diagonal, in the sense of \cite[Definition 2.2]{gimpgofflouctech}), and let $\mu$ be the volume density associated with $\rd$ as in \cite[Section 2]{gimpgofflouctech}. Assume that either
\begin{itemize}
\item $\rd^2$ is smooth on $X\times X$; or
\item $X=G/H$ is a smooth compact homogeneous space, i.e., $G$ is a compact Lie group, and $\rd$ is $G$-invariant.
\end{itemize}
Then, uniformly for $\nu$ in any compact subset of $\mathbb{C}$ we have an asymptotic expansion 
$$\mathcal{M}_X(R,\nu)=\left(\frac{R^n}{n!\omega_n}\right)^{-\nu} \sum_{j=0}^\infty c_j(\nu,X)R^{-2j}+O(R^{-\infty}), \quad \mbox{as $\mathrm{Re}(R)\to \infty$},$$
where 
$$\begin{cases}
c_0(\nu,X)=\mathrm{vol}_n(X), \\
c_1(\nu,X)=-\nu\frac{n+1}{6}\int_X s\rd \mu,\\
c_j(\nu,X)=\int_X p_j(\nu, x)\rd \mu(x), \; j>1.
\end{cases}$$
The function $s$ is the scalar curvature, if $\rd$ is the geodesic distance, and in general is defined from $\rd$ as in \cite[Theorem 6.1]{gimpgofflouctech}. The expression $p_j$ is a universal polynomial in $\nu$ and the covariant derivatives of the Taylor expansion of $\rd^2$ along the diagonal. The polynomial $p_j$ is homogeneous of order $2j$ in  the curvature and its derivatives, where the curvature has order $2$ and its $l$-th derivative is of order $2+l$. 
\end{thm}

Note that in the special case when $G/H$ is a compact homogeneous space, the existence of an asymptotic expansion also follows from \cite[Proposition 3.2]{gimpgofflouctech} and Theorem \ref{mainthm2}.

\begin{proof}[Proof of Theorem \ref{thm3}]
Arguing as in \cite[Section 3]{gimpgofflouctech}, we see that 
\begin{equation}
\label{aldkandlkn}
\mathcal{M}_X(R,\nu)= {R^\nu} \langle Q_X(R)^\nu1,1\rangle_{L^2(X,\mu)}+O(R^{-\infty}),\quad\mbox{as $R\to \infty$},
\end{equation}
where $Q_X$ is the elliptic pseudodifferential operator with parameter defined by
$$Q_X(R)f(x):={R^{-1}}\int_X \chi(x,y)\mathrm{e}^{-R\rd(x,y)}f(y)\rd \mu(y).$$
Here $\chi\in C^\infty(X\times X)$ is a a function such that $\chi(x,y)=1$ sufficiently close to the diagonal $x=y$ and $\chi(x,y)=0$ on the singular support of $\rd^2$. {$Q_X$ thus localizes the integral kernel of $\mathcal{Z}_X$ in a neighborhood of the diagonal $\{(x,x) \in X\times X\}$. The localization affects the asymptotic expansion only by terms of $O(R^{-\infty})$ and therefore implies \eqref{aldkandlkn}. More precisely, if  $\rd^2$ is smooth on $X\times X$ this follows from \cite[Corollary 3.5]{gimpgofflouctech}, while if $G/H$ is a compact homogeneous space this is proven as \cite[Proposition 3.2]{gimpgofflouctech}. }

{For $Q_X$ we may now apply standard techniques for elliptic pseudodifferential ope\-rators with parameter. From \cite[Chapter II]{shubinbook} we obtain that the complex powers $Q_X(R)^\nu$ of $Q_X(R)$ form a holomorphic family in $\nu$ of elliptic pseudodifferential operators with parameter $R$ and that $Q_X(R)^\nu$ depends holomorphically on $R$.}

{The existence of an asymptotic expansion 
$$\mathcal{M}_X(R,\nu)=\left(\frac{R^n}{n!\omega_n}\right)^{-\nu} \sum_{j=0}^\infty c_j(\nu,X)R^{-2j}+O(R^{-\infty}) ,\quad \mbox{as $\mathrm{Re}(R)\to \infty$},$$
then follows from general properties of elliptic pseudodifferential operators with parameter, as in \cite[Subsection 6.1]{gimpgofflouctech}. The structural statements are a consequence of the symbolic computations for $Q_X(R)$ in \cite[Section 2]{gimpgofflouctech} and the symbolic construction of the complex powers \cite[Chapter II]{shubinbook}. Invariant theory, combined with \cite[Theorem 11]{will14}, implies that the first two polynomials $p_0, p_1$ are given by $p_0(\nu,x)=1$ and $p_1(\nu,x)=-\nu\frac{n+1}{6} s(x)$.}
\end{proof}

\section{Examples}

We consider two examples that fit in the general framework of the paper: the $p$-adic integers $\mathbb{Z}_p$ and the $n$-dimensional unit sphere $S^n$. While $\mathbb{Z}_p$ is totally disconnected and not a manifold, it forms a compact group and carries a natural distance function. This allows us to study its magnitude function and residues following the approach of Section \ref{lknland} and \cite{will14}. Meanwhile, $S^n$ is a smooth compact homogeneous space and therefore satisfies the assumptions in the earlier sections of this article. We illustrate how our approach recovers results for the magnitude, respectively residues, of $S^n$ in \cite{ohara23,will14}.

\subsection{$p$-adic integers}

For any prime $p$, the $p$-adic integers $\mathbb{Z}_p$ form a compact group with an invariant distance function defined from the valuation $|\cdot|_p$ and admit the Haar measure $\mathrm{d}x$. While the magnitude of $\mathbb{Z}_p$ is well-defined \cite{meckes,meckes2}, little is known about it. Its Brylinski beta function, however, is readily computed. From the relationship obtained in this article, we deduce novel, exotic behavior of the magnitude function. 

Indeed, for the Brylinski beta function we note
\begin{align*}
B_{\mathbb{Z}_p}(z)=\int_{\mathbb{Z}_p} |x|^z_p\mathrm{d}x=\sum_{k=0}^\infty \int_{p^k\mathbb{Z}_p\setminus p^{k+1}\mathbb{Z}_p} |x|^z_p\mathrm{d}x=\sum_{k=0}^\infty\left(\frac{1}{p^k}-\frac{1}{p^{k+1}}\right) p^{-kz}=\frac{p-1}{p-p^{-z}}.
\end{align*}
This formula shows that the poles of $B_{\mathbb{Z}_p}$ are simple and fill up the set 
$$-1+\frac{2\pi i}{\log(p)}\mathbb{Z}.$$
If we write $\mathfrak{m}_{\mathbb{Z}_p}(R)=\int_{\mathbb{Z}_p}\mathrm{e}^{-R|x|_p}\mathrm{d}x$ as above, we deduce from the location of the poles that  $\mathfrak{m}_{\mathbb{Z}_p}(R)$ does not admit an asymptotic expansion for $\mathrm{Re}(R)\to \infty$; indeed, if there was such an expansion, it would have to take the form 
$$\mathfrak{m}_{\mathbb{Z}_p}(R)=\sum_{z\in -1+\frac{2\pi i}{\log(p)}\mathbb{Z}} \Gamma(z)\frac{p-1}{p\log(p)} R^{z}+O(\mathrm{Re}(R)^{-\infty}).$$
This leads to a contradiction, as the latter asymptotic sum makes no sense because $\mathrm{Re}(z)$ does not go to $-\infty$ for $z$ in the set $-1+\frac{2\pi i}{\log(p)}\mathbb{Z}$.

We can, however, describe $\mathfrak{m}_{\mathbb{Z}_p}(R)$ in some detail. Proceeding as above, we compute  
\begin{align*}
\mathfrak{m}_{\mathbb{Z}_p}(R)=&\int_{\mathbb{Z}_p} \mathrm{e}^{-R|x|_p}\mathrm{d}x=\sum_{k=0}^\infty \int_{p^k\mathbb{Z}_p\setminus p^{k+1}\mathbb{Z}_p} \mathrm{e}^{-R|x|_p}\mathrm{d}x=\\
=&\sum_{k=0}^\infty\left(\frac{1}{p^k}-\frac{1}{p^{k+1}}\right) \mathrm{e}^{-Rp^{-k}}=(p-1)\sum_{k=0}^\infty p^{-k-1}  \mathrm{e}^{-Rp^{-k}}.
\end{align*}
While $\mathfrak{m}_{\mathbb{Z}_p}$ does not admit an asymptotic expansion, we note that $c_pR^{-1} \leq \mathfrak{m}_{\mathbb{Z}_p}(R) \leq C_pR^{-1}$ when $R>1$, for some constants $c_p,C_p>0$. This decay and the real part $-1$ of the poles of $B_{\mathbb{Z}_p}$ both reflect that the Minkowski dimension of $\mathbb{Z}_p$ is $1$ \cite{meckes}.  

\subsection{Spheres}

We consider the $n$-dimensional unit sphere $S^n\subseteq \R^{n+1}$ with its Euclidean distance function $\rd(x,y)=\|x-y\|$ and normalized volume measure $\mu$. Note that $S^n=O(n+1)/O(n)$ is a compact homogeneous space. On the one hand, the Brylinski beta function was computed in \cite[Example 2.3]{FV}:
\begin{equation}
\label{betaforspherh}
B_{S^{n}}(z)=2^{z+n}\pi^{-1/2} \frac{\Gamma\left(\frac{z+n}{2}\right)\Gamma\left(\frac{n+1}{2}\right)}{\Gamma\left(\frac{z}{2}+n\right)}.
\end{equation}
The reader should beware that we are using the normalized measure, while \cite{FV} do not normalize and use the standard measure induced on a hypersurface. 

On the other hand, the proof of \cite[Theorem 13]{will14} and Proposition \ref{mxandmx} show that for any $N\in \mathbb{N}$
$$\mathcal{M}_{S^n}(R,1)\equiv\mathfrak{m}_{S^n}(R)=\sum_{j=0}^N \frac{(-1)^j}{4^j}\begin{pmatrix} \frac{n}{2}-1\\ j\end{pmatrix} (n-1+2j)!R^{-n-2j} +O(R^{-n-2N-1}).$$
In the notation of Section \ref{sectionhomos}, we set 
$$a_{2j}:=\frac{(-1)^j}{4^j}\begin{pmatrix} \frac{n}{2}-1\\ j\end{pmatrix} (n-1+2j)!$$
and $a_{2j+1}=0$. Proposition \ref{mobvingpoweow} then implies the asymptotic expansion
$$\mathcal{M}_{S^n}(R,\nu)\sim \sum_{j=0}^\infty ((n-1)!)^\nu g_j\left(\nu,\frac{a_2}{(n-1)!},\frac{a_4}{(n-1)!},\ldots, \frac{a_{2j}}{a_0}\right)R^{-\nu n-2j}$$
 for $R\to +\infty$, with the polynomial $g_j=g_j(t_1,\ldots, t_j)$  given by
$$g_j(\nu,t_1,\ldots, t_j)=\sum_{\substack{\mathbbm{k}=(k_1,\ldots,k_j)\in \mathbb{N}^{j},\\ \sum_l lk_l=j}}\frac{\nu(\nu-1)\cdots (\nu-\sum_{l=1}^j k_l+1)}{\prod_{l=1}^j k_l!}\prod_{l=1}^j t_l^{k_l}.$$
From Remark \ref{aldnlknad} we thus conclude  
$$B_{S^n}(z)=\sum_{j=0}^N \frac{(-1)^j}{4^j}\begin{pmatrix} \frac{n}{2}-1\\ j\end{pmatrix}\frac{1}{z+n+ 2j} + \tilde{f}_N(z),$$  
where $\tilde{f}_{N}$ is holomorphic in the half-plane $\mathrm{Re}(z)> -n-2N-1$. This reproduces the poles and residues of the exact expression  \eqref{betaforspherh}.

\begin{appendix}

\section{Asymptotics of Mellin transforms}

The next result can be found in  \cite[Proposition 5.1]{gs} or in \cite[Proposition 3.6]{gimpgofflouc}. In the body of this article it is applied to the function $e=\mathfrak{m}_X$.

\begin{prop}
\label{relarlandlna}
Assume that $e(t)$ and $e(t^{-1})$ are holomorphic in $$V_{\theta_0} = \{t = re^{i\theta} :\ 2>r>0, |\theta|<\theta_0\} ,$$ for some $\theta_0 \in (0,\frac{\pi}{2})$, and $e(t) = O(|t|^a)$, $e(t^{-1}) = O(|t|^{b})$ for $t\to 0$ in $V_\delta$, for any $0<\delta<\theta_0$ and some $a,b\in \mathbb{R}$. 
Consider the function
$$f(s) = \int_0^\infty t^{s-1} e(t) \ \rd t,$$ 
holomorphic for $\mathrm{Re}(s)>-a$. Then the following properties are equivalent:
\begin{enumerate}
\item[(a)] $e(t)$ and $e(t^{-1})$ have asymptotic expansions for $t \to 0$, 
\begin{align*}
e(t) &= \sum_{j=0}^N a_{j} t^{\beta_j}+O(t^{\beta_{N+1}}), \ \ \mbox{where}\ \ \beta_0 \leq \beta_1\leq \cdots \leq \beta_N\leq \beta_{N+1}\\
e(t^{-1}) &\sim \sum_{j=0}^M  A_{j} t^{\gamma_j}+ O(t^{\gamma_{M+1}}), \ \ \mbox{where}\ \ \gamma_0 \leq \gamma_1\leq \cdots \leq \gamma_M\leq \gamma_{M+1} 
\end{align*}
uniformly for $t \in V_\delta$, for each $0<\delta<\theta_0$.\\
\item[(b)] $f(s)$ is meromorphic on the strip $-\beta_{N+1}\leq \mathrm{Re}(s)\leq \gamma_{M+1}$ with the singularity structure that
$$\Gamma(s) f(s) = \sum_{j=0}^N \frac{a_{j}}{\beta_j+s} + \sum_{j=0}^M \frac{A_{j}}{\gamma_j-s}+a(s),\ $$
where $a$ is holomorphic, and for each real constants $C_1>-\beta_{N+1}, C_2<\gamma_{M+1}$ and each $0<\delta<\theta_0$,
$$|f(s)| \leq C(C_1,C_2,\delta)e^{-\delta |\mathrm{Im}(s)|}, \ \ |\mathrm{Im}(s)|\geq 1, \ C_1\leq |\mathrm{Re}(s)| \leq C_2.$$
\end{enumerate}
Here $\Gamma$ denotes the Gamma function.
\end{prop}

\end{appendix}

\end{document}